\documentclass[aip,jmp,preprint,floatfix,amsmath,amssymb]{revtex4-2}

\usepackage{amsthm}
\usepackage{graphicx}
\usepackage{subcaption}
\allowdisplaybreaks
\theoremstyle{plain}
\newtheorem{theorem}{Theorem}
\newtheorem{lemma}{Lemma}
\newtheorem{corollary}{Corollary}
\newtheorem{proposition}{Proposition}
\theoremstyle{remark}
\newtheorem{remark}{Remark}

\newcommand{\lam}{\lambda}
\newcommand{\kap}{\kappa}

\begin{document}

\title{Nonlinear springs with orbits in radicals: a solvability criterion}

\author{Unnikrishna Pillai}
\email{up240@nyu.edu}
\affiliation{Department of Electrical and Computer Engineering, New York University,
Brooklyn, NY 11201, USA}

\keywords{nonlinear oscillator, perfect-square spring, closed-form orbit, solvability by
radicals, monodromy group, Galois group}

\begin{abstract}
An odd polynomial oscillator has a periodic orbit given by a quadrature. When the zeros of the
energy polynomial are all simple or all of even multiplicity, that quadrature is elementary
exactly when the polynomial is a perfect square, and the time is a sum of inverse tangents.
Recovering the displacement inverts that relation, algebraic when their coefficients are
commensurate, making a closed form a Galois question. This paper identifies the group as the
monodromy group of a rational map, computes its branch data from the parameters, and shows
the branch value at the phase-angle singularity never vanishes, so every decomposition pairs the
amplitudes reciprocally. In the commensurate family with distinct real factors, the orbit is
radical exactly when the primitive inversion degree is at most four, or when reciprocal pairs
leave a reduced spring of that degree. Radical orbits occur at the fifth, ninth, thirteenth and
seventeenth force orders, the last only for reciprocal members, an expression recorded for each. The multiplicity hypothesis cannot be dropped. Factoring the energy polynomial into square
and squarefree parts turns the quadrature into a differential of the third kind on a hyperelliptic
curve, where elementarity is a torsion condition on the divisor of its poles with a vanishing
holomorphic remainder. The first nonsquare elementary member appears at the ninth force order, and the
power pullback inverts in closed form at every such order, giving a radical orbit for a confining
spring of arbitrarily high degree with a sufficient condition for positive coefficients.
\end{abstract}

\maketitle

\section{Introduction}
\label{sec:intro}

A mass on a nonlinear spring with an odd restoring force obeys
\begin{equation}
    \ddot x+\sum_{k=0}^{2r}a_kx^{2k+1}=0 ,
    \label{eq:spring}
\end{equation}
whose highest force term is $x^{4r+1}$. The orbit under study is a central libration with simple turning points
$\pm A$ and positive kinetic-energy factor between them; that motion is periodic and its period is an
integral~\cite[Sec.~12]{arnold}. For a general choice of the coefficients the integral is
hyperelliptic and the motion has no
elementary expression. For a special choice, the perfect-square case recalled in
Section~\ref{sec:family}, the integral collapses to a sum of inverse tangents, and the time becomes an
elementary function of the displacement. The waveform, meaning the displacement as a function of time,
then needs one more step, the inversion of that elementary relation, and that step is an algebraic
equation whose degree grows with the order of the spring.

When is the quadrature elementary? When is the inversion an
algebraic equation? When are the roots of that equation expressible in radicals?
Under the multiplicity hypothesis that the zeros of the energy polynomial are all simple or all of
even multiplicity, the perfect-square criterion answers the first. Within the distinct-factor
perfect-square family, commensurability of the coefficients answers the second. The third is
a question about a Galois group, and it is the subject of this paper.

Exact periodic solutions of polynomial oscillators are classically obtained through elliptic
functions, as for the cubic and the cubic-quintic Duffing oscillators~\cite{nayfeh,kovacic,elias,belendez},
and the reduction of hyperelliptic integrals to elliptic and elementary ones is tabulated in the
handbooks~\cite{byrd,whittaker}; the present family is the one in which the time integral is
elementary outright. The classical theory of rational functions invertible in radicals begins with Ritt~\cite{rittrad} and
is developed from a modern monodromy viewpoint by Burda~\cite{burda} and in the topological Galois
theory of Khovanskii~\cite{khovanskii}, whose origin is Klein's treatment of the quintic~\cite{klein}. Exact inversion of oscillator
time integrals has also been studied by approximation and quintic reduction~\cite{boschi}. This paper
treats the oscillator-derived rational family: its complete branch data, the reciprocal
decomposition forced by its pole geometry, and the resulting bound on the force order.

The answer is a recursive criterion. A perfect-square member with distinct real factor parameters is called
\emph{commensurate} when its coefficients satisfy $h_k=n_kh_0$ for a nonzero $h_0$ and a primitive
integer vector $(n_1,\ldots,n_r)$. Write $D=\sum_k|n_k|$ for its primitive inversion degree, and call
the member reciprocal when its amplitudes, defined in Section~\ref{sec:family}, come in pairs with
product one carrying equal weights. Radical solvability means solvability of the inversion polynomial
over $\mathbb C(\sigma)$, where $\sigma=\tan(u/h_0)$.

\begin{theorem}
\label{thm:main}
A commensurate distinct-factor member of the perfect-square family has an inversion expressible in
radicals over $\mathbb C(\sigma)$ if and only if one of the following holds.
\begin{enumerate}
\item[\textup{(a)}] The primitive inversion degree satisfies $D\le4$.
\item[\textup{(b)}] The member is reciprocal and its reduced member has primitive inversion degree at
most four, the reduced member being the spring with half as many factor parameters obtained from
the substitution of Section~\ref{sec:reciprocal}, whose inversion equation has half the degree.
\end{enumerate}
\end{theorem}

The statement is proved for every number of distinct nonzero factor parameters. The one
configuration the degree count leaves open, a family whose members carry two fewer amplitudes than
a power of two, is excluded in Section~\ref{sec:general} by a short argument on the cycle structure
of affine maps.

Because a reduced member has softening factor parameters only, it is never itself reciprocal, so the
recursion in (b) stops after one step. Counting degrees then gives the list of orders.

\begin{corollary}
\label{cor:orders}
Within the commensurate distinct-factor perfect-square family, radical orbits exist at the force
orders $x^5$, $x^9$, $x^{13}$ and $x^{17}$, and at no other order of the form $4r+1$. At $x^{17}$
they exist only for reciprocal members. This statement does not classify elementary
mixed-multiplicity energy polynomials. Among springs whose energy-polynomial zeros are all simple or
all of even multiplicity, orders of the form $4r+3$ have no elementary quadrature.
\end{corollary}

The method is classical in outline and explicit in detail. The inversion equation defines a rational
map, and the Galois group of the equation over the field of rational functions of the right-hand side
is the monodromy group of that map. The roots are expressible in radicals exactly when the monodromy
group is solvable, which is the topological form of the Abel--Ruffini theorem developed by
Arnold~\cite{alekseev}. What makes the family tractable is that its branch data can be written
down in closed form: the Wronskian of the inversion equation factors, and Lemma~\ref{lem:wronskian}
below gives the location and multiplicity of every critical point.

The step that makes the classification complete is Lemma~\ref{lem:positive}. The map has four marked
values, two of them at fixed positions and two at $\pm F(j)$, where $j$ is the imaginary unit and $F$
is the inversion map. Everything about the structure of decompositions turns on whether $F(j)$ can
vanish. Lemma~\ref{lem:positive} shows that it never does, for any member of the family. The proof
writes the imaginary part of the relevant integral as a divided difference of the function
$G(s)=(s-1)^{r-1}R(s)$, where $R$ is the Cauchy transform of an explicit positive measure, and then
shows that $G$ has a positive derivative of order $r-1$, by an integral representation below the
singular point of $R$ and by a series with positive coefficients above it. With that in hand, the
possible decompositions collapse to the reciprocal pairing, and the group theory reduces to the
classification of primitive groups containing a cycle~\cite{jones,dixon}.

Section~\ref{sec:family} sets up the family and the inversion equation, and records a
mixed-multiplicity family whose quadratures are elementary without being squares. Section~\ref{sec:branch}
derives the branch data and Section~\ref{sec:positive} proves the positivity lemma.
Section~\ref{sec:structure} proves the two structural lemmas, on equal weights and on decompositions.
Section~\ref{sec:reciprocal} constructs the reduced member of a reciprocal spring.
Section~\ref{sec:groups} carries out the group-theoretic step and proves Theorem~\ref{thm:main}.
Section~\ref{sec:explicit} gives the closed forms at the four solvable orders.
Section~\ref{sec:general} rules out the last configuration, and
Section~\ref{sec:remarks} collects the consequences, including the fact that no substitution of a
power of the displacement can improve matters. Section~\ref{sec:beyond} places the mixed-multiplicity
members: the quadrature becomes a differential of the third kind on a hyperelliptic curve, the
Liouville criterion becomes a torsion condition on the divisor of its poles, and within the orders
congruent to one modulo four the ninth is the lowest at which a nonsquare elementary member occurs.

\section{The family and the inversion equation}
\label{sec:family}

\emph{The energy integral.} Let $V(x)=\sum_ka_kx^{2k+2}/(2k+2)$ be the potential of \eqref{eq:spring}
and let $A$ be the turning point of the orbit under study. Multiplying \eqref{eq:spring} by $\dot x$
and integrating gives $\tfrac12\dot x^2+V(x)=V(A)$. Normalize by $x=Ay$ and $u=\sqrt\alpha\,t$ with
$\alpha=2V(A)/A^2$. Then
\begin{equation}
    \Bigl(\frac{dy}{du}\Bigr)^{2}=G_0(y^2),\qquad
    G_0(w)=\frac{2\bigl[V(A)-V(A\sqrt w)\bigr]}{A^2\alpha},
    \label{eq:energyG}
\end{equation}
where $G_0$ is a polynomial of degree $2r+1$ with $G_0(0)=1$ and $G_0(1)=0$. Factoring out the turning
point,
\begin{equation}
    \Bigl(\frac{dy}{du}\Bigr)^{2}=(1-y^2)\,P(y^2),\qquad P(0)=1,\quad\deg P=2r .
    \label{eq:energy}
\end{equation}
Throughout, $P(w)>0$ for $0\le w\le1$, which is the stated central libration. If ``restoring spring'' is
also required to mean positive linear stiffness at the origin, then differentiating
\eqref{eq:energy} gives the additional condition $a_0/\alpha=1-P'(0)>0$. With $w=y^2$ the
quadrature of \eqref{eq:energy} reads
\begin{equation}
    du=\frac{dw}{2\sqrt{w(1-w)P(w)}} ,
    \label{eq:clock}
\end{equation}
and the question is when this differential has an elementary integral.

\emph{The perfect-square criterion.} Separating variables in \eqref{eq:energy} and substituting
$y=\sin\theta$ gives
\begin{equation}
    u=\int_0^\theta\frac{d\vartheta}{\sqrt{P(\sin^2\vartheta)}} .
    \label{eq:quadtheta}
\end{equation}
The integrand is rational in $\sin\vartheta$ exactly when $P$ is a perfect square. The classification
below treats the squares that split into distinct real factors,
\begin{equation}
    P(w)=\prod_{k=1}^r(1+\lam_kw)^2,\qquad
    \lam_1>\lam_2>\dots>\lam_r>-1,\quad \lam_k\ne0 ,
    \label{eq:perfectsquare}
\end{equation}
with distinct nonzero factor parameters, so that $P$ stays positive on the swing. A square carrying a
pair of nonreal conjugate factors lies outside that family. The zeros of $P$
are $-1/\lam_k$, not the parameters $\lam_k$. If the zeros of $P$ are
all simple, then $(1-y^2)P(y^2)$ has $4r+2$ simple zeros, its square root defines a smooth
hyperelliptic curve of genus $2r$, and the quadrature differential $du$ of \eqref{eq:energy} is a
nonzero holomorphic differential on that curve. A nonzero holomorphic differential has no elementary
primitive by Liouville's theory~\cite{ritt48,rosenlicht,bronstein}. If every root has even multiplicity
then $P$ is a square. A mixed multiplicity puts poles on that curve, so the differential is no longer
holomorphic and the argument does not apply. A factor parameter below $-1$ would put a double zero of $P$ inside the swing,
which is an equilibrium of \eqref{eq:spring}, and the orbit through the origin would approach it with
infinite period instead of returning. Since an odd-degree polynomial is never a square, orders of the
form $x^{4r+3}$ have no elementary member with roots all simple or all of even multiplicity, which is
the last statement of Corollary~\ref{cor:orders}.

\emph{Mixed multiplicities.} Members with mixed multiplicities can still be elementary. The energy
polynomial
\begin{equation}
    P(w)=\Bigl(1+\tfrac{w}{6}\Bigr)^{2}\Bigl(1+\tfrac{w}{27}\Bigr)^{2}
          \Bigl(1+\tfrac{w}{15}\Bigr)\Bigl(1+\tfrac{w}{24}\Bigr)
    \label{eq:mixed}
\end{equation}
carries multiplicities two, two, one and one, stays positive on the swing, and gives a restoring
force of degree thirteen with every coefficient positive. It is not a square, and its quadrature is
elementary all the same. Set $T(w)=(w+6)^2(w+27)$, for which $T(w)-972=w(w+15)(w+24)$ and
$1372-T(w)=(1-w)(w+20)^2$. Those two identities give
\begin{equation}
    \frac{dw}{2(w+6)(w+27)\sqrt{w(1-w)(w+15)(w+24)}}
    =\frac{dT(w)}{6\,T(w)\sqrt{\bigl(T(w)-972\bigr)\bigl(1372-T(w)\bigr)}} ,
    \label{eq:mixedpull}
\end{equation}
and \eqref{eq:quadtheta} is a constant multiple of the integral of either side. A quadratic sits
under the root on the right, so that curve is a conic and the cubic map $T$ pulls the differential
back from genus zero. The angle--time relation is the single inverse tangent
\begin{equation}
    u=2\sqrt{\tfrac{135}{686}}\,\tan^{-1}\!\left[\sqrt{\tfrac{343}{243}}\,
      \frac{y\sqrt{(y^{2}+15)(y^{2}+24)}}{(y^{2}+20)\sqrt{1-y^{2}}}\right] ,
    \label{eq:mixedclock}
\end{equation}
and the orbit has period $4\pi\sqrt{135/686}$. The example is not isolated: with
$T(w)=(w+p)^2(w+q)$ the two identities hold for every $p$ when $q=4p+3$, taking the form
\begin{equation}
    T(w)-p^2q=w\bigl(w^2+(6p+3)w+9p^2+6p\bigr),\qquad
    4(p+1)^3-T(w)=(1-w)(w+3p+2)^2 .
    \label{eq:mixedfam}
\end{equation}
The energy polynomial of the corresponding member is
\begin{equation}
    P(w)=\frac{(w+p)^2(w+q)^2\bigl(w^2+(6p+3)w+9p^2+6p\bigr)}{p^2q^2\bigl(9p^2+6p\bigr)} ,
\end{equation}
normalized to $P(0)=1$. For rational $p$ the quadratic factor in \eqref{eq:mixedfam} splits over the
rational field exactly when $12p+9$ is a square there. Every coefficient of the restoring force stays
positive for $p$ above the largest real root of $4p^3-7p^2-14p-5$, which is the only positive root
of that cubic. Once poles are present a
positive genus no longer obstructs an elementary quadrature, and the perfect-square condition stays
sufficient at every order while its converse needs the multiplicity hypothesis.

\emph{The angle flow and the quadrature.} With \eqref{eq:perfectsquare} the quadrature is
\begin{equation}
    u=\int_0^\theta\frac{d\vartheta}{Q(\vartheta)},\qquad
    Q(\vartheta)=\prod_{k=1}^r\bigl(1+\lam_k\sin^2\vartheta\bigr),
    \label{eq:quad}
\end{equation}
and the waveform is $x(t)=A\sin\theta(\sqrt\alpha\,t)$. With $\eta=\tan\theta$,
$1+\lam_k\sin^2\vartheta=(1+\kap_k^2\eta^2)/(1+\eta^2)$, with
\begin{equation}
    \kap_k=\sqrt{1+\lam_k}>0 ,
    \label{eq:amplitudes}
\end{equation}
the amplitudes of the member, and the quadrature becomes rational,
\begin{equation}
    du=\frac{(1+\eta^2)^{r-1}\,d\eta}{\prod_{k=1}^r\bigl(1+\kap_k^2\eta^2\bigr)} .
    \label{eq:dueta}
\end{equation}
A factor with $\lam_k>0$ hardens the spring and has $\kap_k>1$; a factor with $\lam_k<0$ softens it
and has $\kap_k<1$. Decomposing \eqref{eq:dueta} into partial fractions in $\eta^2$ and integrating term by
term gives
\begin{equation}
    u=\sum_{k=1}^rh_k\tan^{-1}\bigl(\kap_k\eta\bigr),\qquad
    h_k=\frac{\lam_k^{\,r-1}}{\kap_k\prod_{i=1,\,i\ne k}^{r}(\lam_k-\lam_i)} ,
    \label{eq:arctan}
\end{equation}
where the coefficient follows by evaluating the numerator and the remaining factors of the denominator
at the pole $\eta^2=-1/\kap_k^2$ and canceling the common powers of $1+\lam_k$. Letting
$\eta\to\infty$ gives the quarter period as the finite sum $U=\tfrac{\pi}{2}\sum_kh_k$.

\emph{The inversion equation.} Suppose the coefficients are pairwise commensurate,
\begin{equation}
    h_k=n_kh_0,\qquad n_k\in\mathbb{Z},\quad\gcd(n_k)=1 ,
    \label{eq:merge}
\end{equation}
which imposes $r-1$ equations on the $r$ factor parameters. When the constraint Jacobian has full
rank and the solution set is nonempty, a fixed primitive integer pattern therefore gives a locally
one-dimensional family. With $\psi=u/h_0$ and $\sigma=\tan\psi$, and since $\tan^{-1}z=\arg(1+jz)$
for real $z$ and arguments add under multiplication,
\begin{equation}
    X(\eta)+jY(\eta)=\prod_{k=1}^r\bigl(1+j\,\mathrm{sgn}(n_k)\,\kap_k\eta\bigr)^{|n_k|} ,
    \label{eq:XY}
\end{equation}
with $X$ and $Y$ real polynomials, so that $\tan\psi=Y/X$ and the inversion is
\begin{equation}
    Y(\eta)-\sigma X(\eta)=0,\qquad D=\sum_{k=1}^r|n_k| .
    \label{eq:inversion}
\end{equation}
Here and below the gcd means $\gcd(|n_1|,\ldots,|n_r|)$, so the vector is primitive and $D$ is
independent of a common rescaling of $h_0$ and the integers. The base field for radical solvability is
$\mathbb C(\sigma)$. In particular, for $r=1$ the primitive normalization gives $|n_1|=D=1$.
The displacement follows from $\sin\theta=\eta/\sqrt{1+\eta^2}$. Only one root of
\eqref{eq:inversion} is physical; the others are values of the multivalued tangent and are discarded by
continuity from $\eta=0$ at $\psi=0$. Since every $|n_k|\ge1$, the degree satisfies $D\ge r$.

The polynomials $X$ and $Y$ are coprime. A common zero would be a zero of both $X+jY$ and $X-jY$,
which by \eqref{eq:XY} vanish at the points $j/[\mathrm{sgn}(n_k)\kap_k]$ and
$-j/[\mathrm{sgn}(n_\ell)\kap_\ell]$, so such a zero would require
$\mathrm{sgn}(n_\ell)\kap_\ell=-\mathrm{sgn}(n_k)\kap_k$ for some $k$ and $\ell$. The amplitudes
\eqref{eq:amplitudes} are positive and distinct, which leaves no such pair. The leading coefficient of
$X+jY$ in \eqref{eq:XY} is a product of nonzero terms, so the larger of the two degrees is $D$. Write
\begin{equation}
    F=\frac{Y}{X}
    \label{eq:Fdef}
\end{equation}
for the rational map of degree $D$ defined by \eqref{eq:inversion}. It is real, because $X$ and $Y$
are real, and odd, because $\eta\to-\eta$ conjugates the product \eqref{eq:XY}.

\section{Branch data of the inversion map}
\label{sec:branch}

The Galois group of \eqref{eq:inversion} over the field of rational functions of $\sigma$ is the
monodromy group of $F$, the group of permutations of the $D$ roots obtained by carrying $\sigma$ once
around each branch value. The roots are expressible in radicals if and only if that group is
solvable~\cite{alekseev}. Everything therefore rests on the branch data, which the following
lemma supplies in closed form.

\begin{lemma}
\label{lem:wronskian}
For a member, the Wronskian of the inversion equation is
\begin{equation}
    W=XY'-YX'=\frac{1}{h_0}\bigl(1+\eta^2\bigr)^{r-1}\prod_{k=1}^{r}\bigl(1+\kap_k^2\eta^2\bigr)^{|n_k|-1} .
    \label{eq:wronskian}
\end{equation}
Consequently the critical points of $F$ are exactly $\eta=\pm j$, each of multiplicity $r-1$, and
$\eta=\pm j/\kap_k$, of multiplicity $|n_k|-1$; the four points $\sigma=\pm j$ and
$\sigma=\pm F(j)$, called the marked values below, carry every branch value of the map, a marked
value being a branch value exactly when its local permutation is nontrivial; above $F(j)$ and
above $F(-j)$ the local permutation is an $r$-cycle together with $D-r$ fixed points; and above
$\sigma=\pm j$ the cycle lengths are the $|n_k|$.
\end{lemma}

\begin{proof}
From $F=Y/X$ the quotient rule gives $F'=W/X^2$. On the other hand $\sigma=\tan\psi$ gives
$d\sigma/d\eta=(1+\sigma^2)\,d\psi/d\eta$ and $d\psi/d\eta=h_0^{-1}du/d\eta$, so by \eqref{eq:dueta}
\begin{equation}
    W=X^2\bigl(1+\sigma^2\bigr)\frac{1}{h_0}\frac{du}{d\eta}
     =\bigl(X^2+Y^2\bigr)\frac{1}{h_0}\frac{(1+\eta^2)^{r-1}}{\prod_{k=1}^{r}(1+\kap_k^2\eta^2)} .
    \label{eq:wronskian0}
\end{equation}
Multiplying \eqref{eq:XY} by its conjugate gives
$X^2+Y^2=\prod_{k=1}^{r}(1+\kap_k^2\eta^2)^{|n_k|}$, and \eqref{eq:wronskian} follows. Its degree is
$2(r-1)+2\sum_k(|n_k|-1)=2D-2$, which is the largest a Wronskian of degree $D$ can have, so $F$ has no
critical points beyond the zeros of \eqref{eq:wronskian}, and those are as stated. The points
$\pm j/\kap_k$ lie above $\sigma=\pm j$, since \eqref{eq:inversion} reduces there to
$\prod_{k=1}^{r}(1\pm j\,\mathrm{sgn}(n_k)\kap_k\eta)^{|n_k|}=0$. The remaining critical points $\pm j$ lie
above $F(\pm j)$, and $F$ is odd, so $F(-j)=-F(j)$. Since the only critical point above $F(j)$ is
$\eta=j$, with multiplicity $r-1$, the local permutation there is an $r$-cycle and the other $D-r$
sheets are unramified.
\end{proof}

Figure~\ref{fig:branch}(a) shows the critical points, and a formula for $F(j)$ comes next.

\begin{lemma}
\label{lem:Fj}
Let $\beta_k=(1-\kap_k)/(1+\kap_k)$, so that $-1<\beta_k<1$ and $\beta_k\ne0$. Then
\begin{equation}
    F(j)=-j\,\frac{P_-\,-\,P_+}{P_-\,+\,P_+},\qquad
    P_\mp=\prod_{k=1}^{r}\bigl(1\mp\mathrm{sgn}(n_k)\kap_k\bigr)^{|n_k|} ,
    \label{eq:Fj}
\end{equation}
so that $F(j)$ is purely imaginary, $F(j)=\pm j$ is impossible, $F(j)=\infty$ corresponds to the
product below equal to minus one, and
\begin{equation}
    F(j)=0\iff\prod_{k=1}^{r}\beta_k^{\,n_k}=1 .
    \label{eq:Fjzero}
\end{equation}
\end{lemma}

\begin{proof}
Evaluating \eqref{eq:XY} and its conjugate at $\eta=j$ gives $X(j)+jY(j)=P_-$ and $X(j)-jY(j)=P_+$,
hence $X(j)=(P_-+P_+)/2$ and $Y(j)=-j(P_--P_+)/2$, which is \eqref{eq:Fj}. The ratio $P_-/P_+$ is
$\prod_{k=1}^{r}\beta_k^{n_k}$, because a negative weight inverts the corresponding factor, and $F(j)$ vanishes
exactly when that ratio is one. Finally $F(j)=\pm j$ would mean $\tan z=\pm j$ for
$z=u(j)/h_0$, and $\tan z=\pm j$ has no solution, since it gives $\sin z=\pm j\cos z$ and hence
$\sin^2z+\cos^2z=0$.
\end{proof}

The positivity lemma of the next section says that the alternative in
\eqref{eq:Fjzero} never occurs.

\begin{figure*}[t]
\centering
\begin{subfigure}{0.48\textwidth}\centering\includegraphics[width=\linewidth]{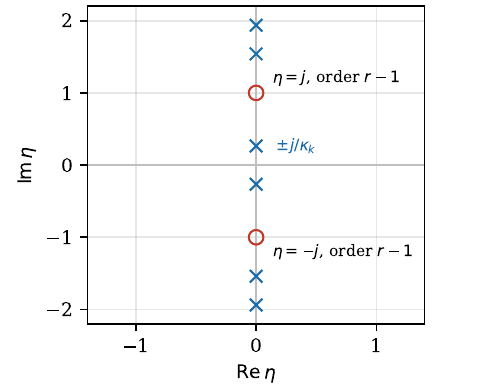}
\caption{}\label{fig:branch_a}\end{subfigure}\hfill
\begin{subfigure}{0.48\textwidth}\centering\includegraphics[width=\linewidth]{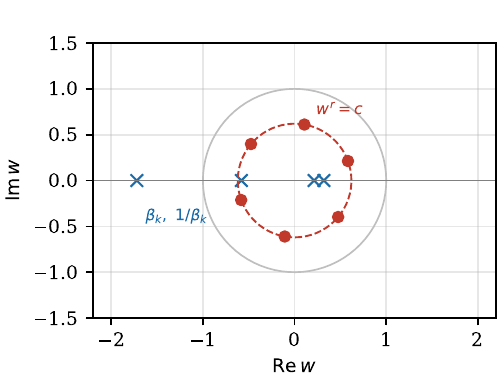}
\caption{}\label{fig:branch_b}\end{subfigure}
\caption{Branch data of the inversion map. (a) Critical points in
the plane of $\eta$: the points $\pm j$ of multiplicity $r-1$, and the points $\pm j/\kap_k$ on the
same axis. (b) The same points in the coordinate $w=(\eta-j)/(\eta+j)$, with the two largest
poles beyond the plotted range. The imaginary axis becomes the
real axis, so the poles of $du$ sit at the real points $\beta_k$ and $1/\beta_k$, while the roots of a
binomial $w^r=c$, drawn here for illustration with six roots, sit at equal angles on a circle, of
which at most two points are real.}
\label{fig:branch}
\end{figure*}

\section{The branch value never vanishes}
\label{sec:positive}

Along the imaginary axis put $\eta=j\tau$ with $0\le\tau<1$. Then \eqref{eq:dueta} reads
\begin{equation}
    du=j\,\frac{(1-\tau^2)^{r-1}\,d\tau}{\prod_{k=1}^{r}\bigl(1-\kap_k^2\tau^2\bigr)} ,
    \label{eq:duimag}
\end{equation}
and the path meets a pole at $\tau=1/\kap_k$ for every hardening factor. Each pole contributes a real
amount to $u(j)$ and the imaginary part is the principal value of \eqref{eq:duimag}. Equivalently,
from \eqref{eq:arctan} with $\tan^{-1}(jz)=j\tanh^{-1}z$ for $|z|<1$ and
$\tan^{-1}(jz)=\pm\pi/2+j\tanh^{-1}(1/z)$ for $z>1$,
\begin{equation}
\begin{split}
    \mathrm{Im}\,u(j)&=\sum_kh_k\,\phi(\kap_k),\\
    \phi(\kap)&=\tanh^{-1}\bigl(\min(\kap,1/\kap)\bigr)>0 .
\end{split}
    \label{eq:Imu}
\end{equation}

\begin{lemma}
\label{lem:positive}
For every perfect-square spring with distinct nonzero factor parameters larger than $-1$,
\begin{equation}
    \mathrm{Im}\,u(j)>0 .
    \label{eq:positive}
\end{equation}
Consequently $F(j)\ne0$ for every member, and the inversion map has four distinct marked values.
\end{lemma}

\begin{proof}
Write $s_k=\kap_k^2=1+\lam_k$, so that the nodes $s_k$ are distinct and positive, and set
\begin{equation}
    R(s)=\frac{\phi(\sqrt s)}{\sqrt s},\qquad G(s)=(s-1)^{r-1}R(s) .
    \label{eq:GA}
\end{equation}
Since $\lam_k=s_k-1$, the products $\prod_{i=1,\,i\ne k}^{r}(\lam_k-\lam_i)$ and
$\prod_{i=1,\,i\ne k}^{r}(s_k-s_i)$ agree, and the coefficient in \eqref{eq:arctan} gives
$h_k\phi(\kap_k)=G(s_k)/\prod_{i=1,\,i\ne k}^{r}(s_k-s_i)$, so that \eqref{eq:Imu} is
the divided difference~\cite{deboor} of order $r-1$,
\begin{equation}
    \mathrm{Im}\,u(j)=G[s_1,\dots,s_r] .
    \label{eq:divdiff}
\end{equation}

\emph{An integral representation.} For $0<s<1$ one has
$R(s)=\int_0^1dt/(1-st^2)$, and for $s>1$ the same integral taken as a principal value equals
$\tanh^{-1}(1/\sqrt s)/\sqrt s$. Substituting $x=1/t^2$ turns it into a Cauchy transform,
\begin{equation}
    R(s)=\int_1^{\infty}\frac{d\mu(x)}{x-s},\qquad d\mu(x)=\tfrac12x^{-1/2}dx ,
    \label{eq:markov}
\end{equation}
again as a principal value when $s>1$. The measure is positive and supported on $[1,\infty)$.

\emph{Positivity of the derivative below the singular point.} For fixed $x\ge1$, dividing the
polynomial $(s-1)^{r-1}$ by $s-x$ gives $(s-1)^{r-1}=q(s)(s-x)+(x-1)^{r-1}$ with $\deg q=r-2$, so
\begin{equation}
    \frac{(s-1)^{r-1}}{x-s}=-q(s)+\frac{(x-1)^{r-1}}{x-s} ,
    \label{eq:division}
\end{equation}
whose derivative of order $r-1$ is $(r-1)!\,(x-1)^{r-1}(x-s)^{-r}$, because the polynomial part has
degree $r-2$. Differentiating \eqref{eq:markov} under the integral sign,
\begin{equation}
\begin{split}
    G^{(r-1)}(s)&=(r-1)!\int_1^{\infty}\frac{(x-1)^{r-1}}{(x-s)^{r}}\,d\mu(x)\\
    &>0,\qquad 0<s<1 ,
\end{split}
    \label{eq:derbelow}
\end{equation}
since every factor in the integrand is positive there.

\emph{Positivity of the derivative above the singular point.} For $s>1$ expand
$R(s)=\sum_{m\ge0}s^{-m-1}/(2m+1)$, which converges, and set $y=1/s$, so that
$G(s)=(1-y)^{r-1}y^{2-r}\sum_{m\ge0}a_my^m$ with $a_m=1/(2m+1)$. Using
$d/ds=-y^2d/dy$ and $(-y^2d/dy)^{r-1}y^p=(-1)^{r-1}p(p+1)\cdots(p+r-2)\,y^{p+r-1}$ term by term, the
coefficient of $y^{n+1}$ in $G^{(r-1)}$ is
\begin{equation}
    (-1)^{r-1}(n+2-r)_{r-1}\sum_{i=0}^{\min(n,r-1)}(-1)^i\binom{r-1}{i}a_{n-i} ,
    \label{eq:coeff}
\end{equation}
where $(z)_p=z(z+1)\cdots(z+p-1)$ and where the index $i$ has dropped out of the Pochhammer factor
because $p=i+m+2-r$ with $m=n-i$. The factor $(n+2-r)_{r-1}$ vanishes for $n\le r-2$, so the series
starts at $n=r-1$. The alternating sum is evaluated by the Beta-function identity
$\sum_i(-1)^i\binom{N}{i}/(z+i)=N!/[z(z+1)\cdots(z+N)]$, applied after writing
$a_{n-i}=1/(2n-2i+1)$ and reversing the order of summation, which gives
\begin{equation}
\begin{split}
    \sum_{i}(-1)^i&\binom{r-1}{i}a_{n-i}\\
    &=\frac{(-1)^{r-1}\,2^{r-1}(r-1)!}{(2n-2r+3)(2n-2r+5)\cdots(2n+1)} .
\end{split}
    \label{eq:beta}
\end{equation}
The two signs cancel, and with $m=n-r+1\ge0$,
\begin{equation}
\begin{split}
    G^{(r-1)}(s)&=\sum_{m\ge0}2^{r-1}(r-1)!\\
    &\times\frac{(m+1)(m+2)\cdots(m+r-1)}{(2m+1)(2m+3)\cdots(2m+2r-1)}\,s^{-m-r}\\
    &>0,\qquad s>1 .
\end{split}
    \label{eq:derabove}
\end{equation}

\emph{From the derivative to the divided difference.} At $s=1$ the function $R$ has a logarithmic
singularity, so $G$ behaves like $(s-1)^{r-1}\log|s-1|$ there. Its derivative of order $r-1$ therefore
has a logarithmic singularity as well, which is locally integrable, and the derivative is positive on
both sides by \eqref{eq:derbelow} and \eqref{eq:derabove}. The nodes themselves avoid the singular
point, because $s_k=1+\lam_k$ and no factor parameter is zero, but
the simplex they span contains it whenever the member has both a hardening and a softening factor.
For $r=1$ the divided difference is simply $G(s_1)>0$, so assume $r\ge2$. More precisely, regard
$T=G^{(r-1)}$ as a distribution. The local form
$G(s)=(s-1)^{r-1}O(\log|s-1|)$ shows that $G,G',\ldots,G^{(r-2)}$ have matching limiting values at
$s=1$; hence the distributional derivative contains no Dirac mass there. It is exactly the locally
integrable function computed in \eqref{eq:derbelow} and \eqref{eq:derabove}, and it is positive almost
everywhere. Mollify $G$ with a nonnegative kernel supported in a small interval. Differentiation
commutes with convolution, so the derivative of order $r-1$ of the mollified function is
$T*\varphi_\epsilon>0$ on the node interval. The Hermite--Genocchi formula writes its divided
difference as the integral of this derivative over the simplex, which is therefore positive. Near the
singular point $T$ is bounded by a constant multiple of
$\bigl|\log|s-1|\bigr|$, which is integrable on the simplex, so dominated convergence carries the
integral to the limit as the kernel shrinks. The divided difference converges with it, being a fixed
finite combination of values at the nodes, where the mollified functions converge uniformly, so
$G[s_1,\dots,s_r]>0$. Strictness follows because $T$ is bounded away from zero on a set of positive
measure.

Finally, $\mathrm{Im}\,u(j)>0$ makes $u(j)/h_0$ a point off the real axis, so
$F(j)=\tan\bigl(u(j)/h_0\bigr)$ is finite and nonzero, the tangent being finite off the real axis,
and by Lemma~\ref{lem:Fj} it is purely imaginary and different from $\pm j$. The four marked
values $\pm j$ and $\pm F(j)$ are therefore distinct. For $r=1$ the primitive degree is one and
no marked value is a branch value, and for $r=2$ with both weights of magnitude one the points
$\pm j$ in the target are regular, so an identity permutation over a marked value is admitted
throughout what follows.
\end{proof}

\begin{remark}
When every factor softens, every $\kap_k<1$, the integrand in \eqref{eq:duimag} has no pole and the
positivity is immediate. The content of Lemma~\ref{lem:positive} is the case of mixed factors, where the
path of integration crosses poles and the integrand changes sign.
\end{remark}

\section{Structural lemmas}
\label{sec:structure}

Throughout this section $F$ is the inversion map of a member. The Cayley coordinate
\begin{equation}
    w=\frac{\eta-j}{\eta+j}
    \label{eq:cayley}
\end{equation}
sends $j$ to $0$, sends $-j$ to $\infty$, and sends the imaginary axis of $\eta$ to the real axis of
$w$. Under it, \eqref{eq:dueta} becomes
\begin{equation}
    du=c\,\frac{w^{r-1}\,dw}{\prod_{k=1}^{r}(w-\beta_k)(w-1/\beta_k)},\qquad
    \beta_k=\frac{1-\kap_k}{1+\kap_k} ,
    \label{eq:duw}
\end{equation}
with a constant $c$, so the poles of $du$ are the real points $\beta_k$ and $1/\beta_k$, and the map
$\eta\to-\eta$ becomes $w\to1/w$.

Figure~\ref{fig:branch}(b) shows those poles and the roots of a binomial in this coordinate.

\begin{lemma}
\label{lem:equal}
If $r\ge3$ then not all $|n_k|$ can equal one, so $D\ge r+1$.
\end{lemma}

\begin{proof}
If every $|n_k|=1$ then $D=r$ and, by Lemma~\ref{lem:wronskian}, the permutations above $\sigma=\pm j$
are trivial, so $F$ is totally ramified at $\eta=\pm j$. A combination of $X$ and $Y$ therefore
vanishes to order $r$ at $\eta=j$ and another at $\eta=-j$, and the two combinations are
independent since otherwise $X$ or $Y$ would vanish to order $2r>D$, so $X+jY$ is a linear
combination of $(\eta-j)^r$ and $(\eta+j)^r$. In the coordinate \eqref{eq:cayley} its roots satisfy $w^r=c$ for a
constant $c$ and lie at equal angles on a circle about the origin. On the other hand those roots are
the points $\eta=j/(\mathrm{sgn}(n_k)\kap_k)$, which are purely imaginary and therefore map to real
values of $w$. A circle about the origin contains at most two real points, so $r\le2$.
\end{proof}

\begin{lemma}
\label{lem:decomp}
If $F=G\circ H$ with both factors of degree at least two, then $\deg H=2$, the deck transformation of
$H$ is $\eta\to-1/\eta$, and the member is reciprocal: its amplitudes come in pairs with product one,
carrying equal weights.
\end{lemma}

\begin{proof}
Every branch value of $G$ is a branch value of $F$. If $G$ were unramified above $\pm F(j)$ its only
branch values would be $\sigma=\pm j$, so $G$ would be branched over two points at most. The
sphere with two points removed has infinite cyclic fundamental group, so a connected covering of it
is the cyclic one, and a rational map branched over at most two points is $z\mapsto z^d$ after
M\"obius changes on source and target~\cite{volklein}. Such a $G$ is totally ramified above
$\sigma=j$, and its degree would divide every
$|n_k|$, against $\gcd(n_k)=1$. So $G$ is ramified at some point $y$ above $F(j)$ or above $F(-j)$.
Every point of $H^{-1}(y)$ is then a critical point of $F$ above the same value, and by
Lemma~\ref{lem:wronskian} the only such points are $j$ and $-j$, so
$H^{-1}(y)\subseteq\{j,-j\}$.

If $H(j)=H(-j)$ then $F(j)=F(-j)=-F(j)$, so $F(j)=0$, which Lemma~\ref{lem:positive}
forbids. So $H(j)\ne H(-j)$, and $H^{-1}(y)$ is a single point, say $\{j\}$, so $H$ is totally
ramified at $j$ and $b\,\nu_G(y)=r$ with $b=\deg H$ and $\nu_G(y)\ge2$, whence $b\le r/2$. The image
$H(-j)$ must then be ramified for $G$ as well, since otherwise the index of $F$ at $-j$ would equal
$\nu_H(-j)=r>b$, which is impossible. Applying the same argument at $-j$ gives
$H^{-1}(H(-j))=\{-j\}$, so $H$ is totally ramified at both points. In the coordinate
\eqref{eq:cayley} the map $H$ is then a power of $w$ up to a M\"obius change on the target, and its
deck rotations $w\to\epsilon w$ with $\epsilon^b=1$ leave $F$ unchanged, hence leave $du$ unchanged, hence
permute the poles of \eqref{eq:duw}. Those poles are real and nonzero, and multiplication by $\epsilon$
maps a finite set of nonzero reals to itself only when $\epsilon=\pm1$. Hence $b=2$ and the deck map is
$w\to-w$, which is $\eta\to-1/\eta$. Under it \eqref{eq:dueta} is unchanged, and comparing the poles
and residues on the two sides pairs each $\kap_k$ with $1/\kap_k$ and each weight with the weight of
its partner.
\end{proof}

\section{Reciprocal members and the reduction}
\label{sec:reciprocal}

The differential \eqref{eq:dueta} is unchanged by $\eta\to-1/\eta$ exactly when the amplitudes come in
pairs with product one carrying equal weights, since that map sends $\kap$ to $1/\kap$, and an
unpaired amplitude would have to satisfy $\kap^2=1$, that is $\lam=0$, which is not an allowed factor
parameter. The pairing therefore needs an even number of factor parameters, $r=2m$. Let the pair sums be
$\rho_i=\kap_i+1/\kap_i>2$ and put $\xi=\eta-1/\eta$. Each pair contributes
\begin{equation}
    \bigl(1+\kap^2\eta^2\bigr)\bigl(1+\eta^2/\kap^2\bigr)=\eta^2\bigl(\xi^2+\rho^2\bigr),
    \label{eq:pairfactor}
\end{equation}
because $\xi^2=\eta^2-2+\eta^{-2}$ and $\rho^2=\kap^2+2+\kap^{-2}$, while $1+\eta^2=\eta(\eta+1/\eta)$
and $d\xi=(\eta+1/\eta)d\eta/\eta$. Substituting in \eqref{eq:dueta} and canceling the powers of
$\eta$,
\begin{equation}
    du=\frac{(\xi^2+4)^{m-1}\,d\xi}{\prod_{i=1}^m\bigl(\xi^2+\rho_i^2\bigr)} ,
    \label{eq:duxi}
\end{equation}
since $(\eta+1/\eta)^2=\xi^2+4$. With $\xi=2\zeta$ this becomes
\begin{equation}
    du=\gamma\,\frac{(1+\zeta^2)^{m-1}\,d\zeta}
       {\prod_{i=1}^m\bigl(1+(2/\rho_i)^2\zeta^2\bigr)},\qquad
    \gamma=\frac{2^{2m-1}}{\prod_{i=1}^m\rho_i^2} ,
    \label{eq:duzeta}
\end{equation}
which after the time rescaling $u\to u/\gamma$ is the quadrature \eqref{eq:dueta} again, for a
spring with $m$ factor parameters and amplitudes $2/\rho_i<1$, all softening, and with half the inversion degree.
The constant $\gamma$ leaves the primitive integer weights and the inversion degree unchanged.
This reduced member is the one referred to in Theorem~\ref{thm:main}. In the angle the substitution is
the doubling
\begin{equation}
    \xi=\tan\theta-\cot\theta=-2\cot2\theta .
    \label{eq:doubling}
\end{equation}
Since a reduced member has softening factor parameters only, it is never itself
reciprocal.

Figure~\ref{fig:red}(a) shows the amplitudes of a reciprocal seventeenth-order member joined in pairs
with product one, beside the amplitudes of an unpaired member at the same order, and
Fig.~\ref{fig:red}(b) shows the angle doubling \eqref{eq:doubling}.

\begin{figure*}[t]
\centering
\begin{subfigure}{0.48\textwidth}\centering\includegraphics[width=\linewidth]{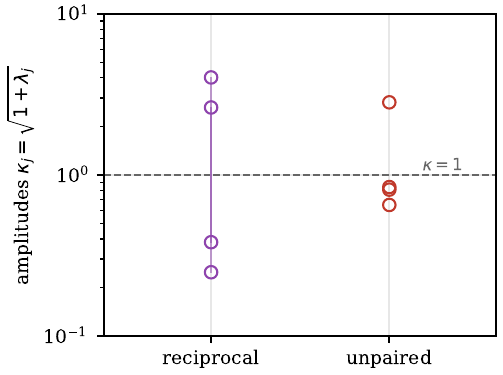}
\caption{}\end{subfigure}\hfill
\begin{subfigure}{0.48\textwidth}\centering\includegraphics[width=\linewidth]{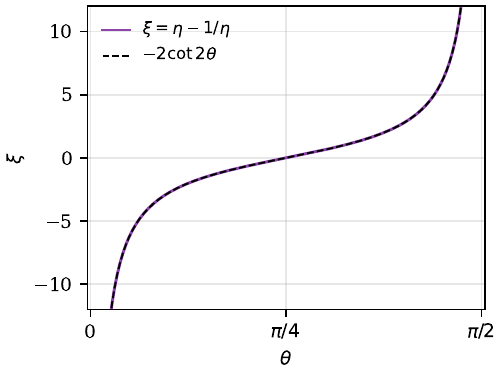}
\caption{}\end{subfigure}
\caption{The reciprocal structure. (a) Amplitudes of a reciprocal seventeenth-order member, joined in
pairs with product one, and of an unpaired member of the same order. (b) The substitution of
Section~\ref{sec:reciprocal} is the doubling of the angle.}
\label{fig:red}
\end{figure*}

The proof of Corollary~\ref{cor:orders} uses two inequalities on the coefficients. Order the factor parameters as in
\eqref{eq:perfectsquare}. Exactly $k-1$ of the factors $\lam_k-\lam_i$ are negative, so their product
has the sign $(-1)^{k-1}$, and $\lam_k^{r-1}$ has the sign of $\lam_k$ when $r$ is even and is positive
when $r$ is odd. The coefficients \eqref{eq:arctan} therefore alternate in sign along the ordered
parameters, except, when $r$ is even, for one repetition at the change of sign of the parameters. For three parameters the pattern is
$(+,-,+)$, and since the coefficients sum to $2U/\pi>0$,
\begin{equation}
    |n_2|<|n_1|+|n_3| .
    \label{eq:h2bound}
\end{equation}
Multiplying each $|h_k|$ by the common positive factor
$|(\lam_1-\lam_2)(\lam_1-\lam_3)(\lam_2-\lam_3)|$ gives, for three parameters,
\begin{equation}
    |h_1|:|h_2|:|h_3|=(\lam_2-\lam_3)\frac{\lam_1^2}{\kap_1}:
    (\lam_1-\lam_3)\frac{\lam_2^2}{\kap_2}:(\lam_1-\lam_2)\frac{\lam_3^2}{\kap_3},
    \label{eq:ratios}
\end{equation}
every factor positive. For softening factor parameters the function $|\lam|^{\,r-1}/\kap$ decreases as $\lam$
increases on $(-1,0)$, since differentiating it gives two non-positive terms. When all three parameters
soften, $\lam_1^2/\kap_1<\lam_2^2/\kap_2$ by that monotonicity, and $\lam_1-\lam_3>\lam_2-\lam_3$,
so $|n_1|<|n_2|$. No pattern with $|n_1|+|n_2|+|n_3|\le4$ satisfies both
\eqref{eq:h2bound} and this inequality, which is the statement used in the proof of
Corollary~\ref{cor:orders}.

\section{The group-theoretic step}
\label{sec:groups}

Decompositions of $F$ correspond to systems of blocks of its monodromy group, so $F$ is indecomposable
exactly when the group is primitive~\cite{ritt22}. By Lemma~\ref{lem:wronskian} the group contains an
$r$-cycle with $f=D-r$ fixed points, the permutation above $F(j)$, and a permutation above $\sigma=j$
whose cycle lengths are the $|n_k|$ and whose number of cycles is therefore exactly $r$.

The monodromy groups that arise from rational maps of the sphere form a restricted
class~\cite{guralnick}. The following facts about primitive groups are used.
A primitive group containing a
transposition is the
full symmetric group, and a primitive group containing a $3$-cycle contains the alternating
group~\cite{dixon}. A primitive group containing a cycle with at least three fixed points contains the
alternating group~\cite{jones}. A solvable primitive group has prime-power degree and is contained in
the affine group of that degree, and the fixed points of an affine map form an affine
subspace~\cite{dixon}.

\begin{lemma}
\label{lem:primitive}
Let $F$ be indecomposable with $D\ge5$ and let its monodromy group be solvable. Then $D$ is a power of
two, $D=r+2$, and the weights consist of two entries of magnitude two and $r-2$ entries of magnitude
one.
\end{lemma}

\begin{proof}
The group is primitive and contains an $r$-cycle with $f=D-r$ fixed points. If $f\ge3$ the group
contains the alternating group of degree $D\ge5$ and is not solvable~\cite{jones}. If $f=0$ then
$D=r$, which Lemma~\ref{lem:equal} excludes for $r\ge3$, while $r\le2$ gives $D\le2$, leaving
$f=1$ and $f=2$.

For $f=1$, $r=D-1$, so the permutation above $\sigma=j$ has $D-1$ cycles on $D$ points, which
forces one cycle of length two and $D-2$ fixed points, a transposition. A primitive group containing a
transposition is the symmetric group of degree $D$, which is not solvable for $D\ge5$.

For $f=2$, $r=D-2$ and the permutation above $\sigma=j$ has $D-2$ cycles, so it is either a
$3$-cycle with $D-3$ fixed points or a product of two transpositions. A $3$-cycle in a primitive group
brings in the alternating group, again excluding solvability for $D\ge5$. In the remaining case the
group is solvable and primitive, hence affine of degree $q^{\ell}=D$ on a space of $\ell$
coordinates over a field with $q$ elements, and the fixed points of the $r$-cycle form an affine
subspace whose size is a power of $q$ and equals two, so $q=2$ and $D=2^{\ell}$. The weights are two
entries of magnitude two and $r-2$ entries of magnitude one, since their sum is $D=r+2$ and their
number is $r$.
\end{proof}

\begin{lemma}
\label{lem:exceptional}
The configuration that Lemma~\ref{lem:primitive} leaves open does not occur, at any degree.
\end{lemma}

\begin{proof}
The case requires a member whose monodromy group of degree $D=2^{\ell}$ contains the $r$-cycle above
$F(j)$, which is a single cycle of length $r=D-2$ moving all but two of the $D$ points. That cycle is
a permutation induced by the group, and the group is affine, so the cycle is realized by an affine map
$g(x)=Mx+v$ of the space of $\ell$ coordinates over the field of two elements, with $M$ its linear
part.
Two properties of $g$ are in conflict.

First, $g$ fixes two points. If $g(p)=p$ then $v=(I-M)p$, and for every integer $k$,
\begin{equation*}
\begin{split}
g^{k}(p)&=M^{k}p+(I+M+\dots+M^{k-1})v\\
&=M^{k}p+(I-M^{k})p=p,
\end{split}
\end{equation*}
so the powers of $g$ fix $p$ as well. Taking $k$ equal to the order $m$ of $M$, where $M^{m}=I$, gives
$g^{m}(x)=x+(I+M+\dots+M^{m-1})v$, a pure translation that fixes $p$, hence the zero translation, hence
the identity. The order of $g$ then divides $m$, the order of its linear part, and in fact equals
it.

Second, $g$ is a single cycle of length $D-2$, so its order is at least $D-2$. Combining the two,
$m\ge D-2$. At $D=4$ the inversion equation has degree four and the member is solvable by (a) of
Theorem~\ref{thm:main}, so let $D\ge8$. The linear part $M$ fixes the difference of the two fixed
points of $g$, which is a nonzero vector. In a basis whose first vector is that one, $M$ is block
triangular, carrying the entry $1$, a row $b^{\top}$, and a block $N$ of dimension $\ell-1$, and the
power $M^{k}$ carries the entry $1$, the row $b^{\top}(I+N+\dots+N^{k-1})$, and the block $N^{k}$.
Write $n_0$ for the order of $N$. The matrix $M^{n_0}$ then differs from the identity only in that
row, and it squares to the identity, so $m$ is $n_0$ or $2n_0$. The largest order of an element of
the general linear group of dimension $\ell-1$ over the field of two elements is $D/2-1$, attained
by a Singer cycle~\cite{darafsheh}, so $m\le D-2$, with equality only when $n_0=D/2-1$ and
$m=2n_0$. An element of that order generates a Singer cycle and fixes no nonzero vector, since its
odd order makes it semisimple and a reducible semisimple element has order at most the least
common multiple of $2^a-1$ and $2^b-1$ with $a+b=\ell-1$, which is below $2^{\ell-1}-1$; so $N-I$ is
invertible. Writing $S=I+N+\dots+N^{n_0-1}$, the identity $(N-I)S=N^{n_0}-I=0$ gives $S=0$, hence
$M^{n_0}=I$ and $m=n_0$. So $m<D-2$, which contradicts $m\ge D-2$. No such $g$ exists, so the
open case is empty.
\end{proof}

\begin{proof}[Proof of Theorem~\ref{thm:main}]
For sufficiency, if $D\le4$ the Galois group has degree at most four and is solvable. If the
member is reciprocal, Section~\ref{sec:reciprocal} shows that $F$ factors through the degree-two map
$\eta\to\eta-1/\eta$ and that the outer factor is the inversion map of the reduced member composed
with a M\"obius map of the target, which leaves the monodromy group unchanged because it leaves
the field $\mathbb C(\sigma)$ unchanged; the group of
$F$ is then an extension of the group of the reduced map by an abelian group, and an extension of a
solvable group by an abelian group is solvable.

For necessity, let the group of a member be solvable and let $D\ge5$, with $r\ge2$. If $F$ is decomposable
then by Lemma~\ref{lem:decomp} the member is reciprocal, and its reduced member has the group of the
outer factor, which is a quotient of a solvable group, hence solvable. Every amplitude of the reduced
member is less than one, so that member is not reciprocal, and Lemma~\ref{lem:decomp} applied to it
makes its inversion map indecomposable. If its primitive degree were at least five,
Lemmas~\ref{lem:primitive} and~\ref{lem:exceptional} would leave it no solvable group. That degree is
therefore at most four, which is case (b). If $F$ is
indecomposable then Lemma~\ref{lem:primitive} applies, and Lemma~\ref{lem:exceptional} removes the
remaining case, so no such member exists. Together with $D\le4$ for the small degrees, this exhausts
the possibilities. Notice that $r=1$ gives $D=1$ under the primitive convention and is already
included.
\end{proof}

\begin{proof}[Proof of Corollary~\ref{cor:orders}]
A reduced member has softening factor parameters only, by Section~\ref{sec:reciprocal}, and a pair with product one
needs one hardening and one softening amplitude, so a reduced member is never reciprocal and the
recursion in Theorem~\ref{thm:main} stops after one step. For $r=1$ the primitive degree is one, giving
order $x^5$. For $r=2$ the reciprocal members with two amplitudes of product one have $D=2$,
giving $x^9$. For $r=3$, Lemma~\ref{lem:equal} gives $D\ge4$, and a member with $D=4$ is
constructed in Section~\ref{sec:explicit}, hence $x^{13}$. For $r=4$, Lemma~\ref{lem:equal} gives
$D\ge5$, so an unpaired member is not solvable; the reciprocal members constructed in
Section~\ref{sec:explicit} reduce to
degree three and gives $x^{17}$. If $r\ge5$ is odd, reciprocal pairing is impossible and
$D\ge r+1>4$. If $r\ge6$ is even and the member is reciprocal, its reduced member has $m=r/2$
softening factor parameters. For $m=3$ the inequalities of Section~\ref{sec:reciprocal} exclude
reduced degree four, and for $m\ge4$ Lemma~\ref{lem:equal} gives reduced degree at least $m+1\ge5$.
Thus no larger $r$ meets either alternative of Theorem~\ref{thm:main}. This proves the claimed list
within the class stated in the corollary.
\end{proof}

Figure~\ref{fig:ladder} summarizes the outcome, and Figure~\ref{fig:mono} shows the monodromy of two
members of the same degree at the seventeenth order, one solvable and one not.

\begin{figure*}[t]
\centering
\includegraphics[width=0.52\textwidth]{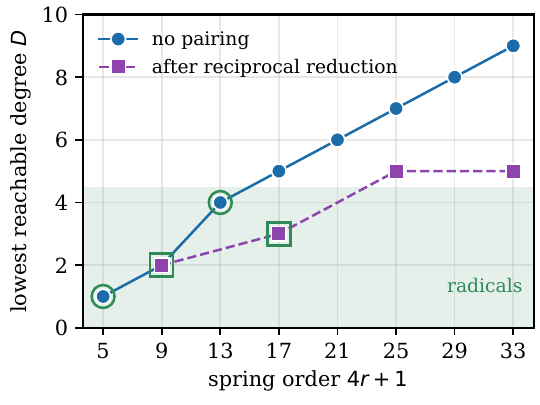}
\caption{Within the commensurate distinct-factor perfect-square family, the lower bound on the
primitive degree of a member of each order, without the reciprocal reduction and with it, attained
at the orders shown with outlined markers. The shaded band is the region where the Galois group is solvable for
every branch datum. The outlined markers are the orders that reach it, which are the orders with
radical inversions.}
\label{fig:ladder}
\end{figure*}

\begin{figure*}[t]
\centering
\begin{subfigure}{0.48\textwidth}\centering\includegraphics[width=\linewidth]{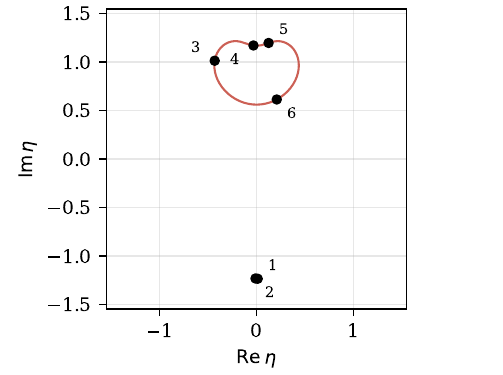}
\caption{}\end{subfigure}\hfill
\begin{subfigure}{0.48\textwidth}\centering\includegraphics[width=\linewidth]{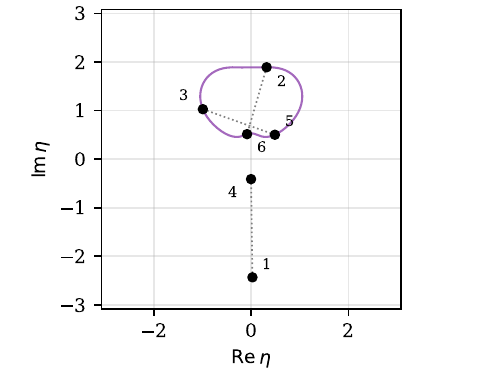}
\caption{}\end{subfigure}
\caption{Roots of the inversion equation of two seventeenth-order members of degree six, carried once
around the branch value $F(j)$. Dots mark the starting positions and the numbers label them. (a) An
unpaired member with weights $(2,-2,1,1)$: the displayed loop gives a four-cycle and two fixed roots,
which are mirror images across the imaginary axis and overlap in the plot.
(b) A reciprocal member with weights $(2,-1,-1,2)$: the dotted segments join the roots exchanged by
$\eta\to-1/\eta$. This single-loop plot illustrates the local permutations only; the solvability
statements follow from Theorem~\ref{thm:main}, not from the plot alone.}
\label{fig:mono}
\end{figure*}

\section{The closed forms}
\label{sec:explicit}

This section records the waveforms at the four solvable orders. In each case $\psi=u/h_0$ and
$\sigma=\tan\psi$, and the displacement is $x=A\eta/\sqrt{1+\eta^2}$ with $\eta$ the physical root.
The variable $\sigma$ does not retain the lift of $\psi$ modulo $\pi$. The physical root is therefore
selected by continuity from $\eta=0$ at the beginning of a quarter period, using the strictly
increasing original relation $u(\eta)$, and the remaining orbit is obtained by reflection and
oddness. At a tangent pole, at $\sigma=0$ on a different lift, or when an intermediate radical
denominator vanishes, the value means the continuous limit of that branch (equivalently, the root of
the lower-degree limiting equation). This prescription includes the turning-point limit
$\eta\to\infty$, $x\to A$, and prevents a change to a nonphysical algebraic branch.

\emph{Fifth order.} Here $r=1$, $h_0=1/\kap$ and the inversion is $\eta=\tan(\kap u)/\kap$, so
\begin{equation}
    x(t)=A\,\frac{\sin(\kap u)}{\sqrt{1+\lam\cos^2(\kap u)}},\qquad u=\sqrt\alpha\,t ,
    \label{eq:x5wave}
\end{equation}
with quarter period $U=\pi/(2\kap)$.

\emph{Ninth order.} Here $r=2$, and the lowest degree $D=2$ requires $|h_1|=|h_2|$, which by
\eqref{eq:arctan} means $|\lam_1|/\sqrt{1+\lam_1}=|\lam_2|/\sqrt{1+\lam_2}$. The function
$|\lam|/\sqrt{1+\lam}$ is strictly increasing for $\lam>0$ and strictly decreasing on $-1<\lam<0$, so
one factor hardens and the other softens. Writing $\lam_1=\lam$ and $\lam_2=-\nu$ the condition squares
to $(1+\lam)\nu^2+\lam^2\nu-\lam^2=0$, whose only root in $(0,1)$ is
\begin{equation}
    \nu=\frac{\lam}{1+\lam},\qquad\text{equivalently}\qquad(1+\lam)(1-\nu)=1 ,
    \label{eq:x9balance}
\end{equation}
so the amplitudes satisfy $\kap_1\kap_2=1$ and the member is reciprocal with $n=(1,1)$. The tangent
addition law gives $\sigma\eta^2+\rho\eta-\sigma=0$ with $\rho=\kap+1/\kap$, and dividing by
$\sigma\eta$ turns it into $\eta-1/\eta=-\rho/\sigma$, that is
\begin{equation}
    \tan2\theta=\frac{2}{\rho}\tan\psi ,
    \label{eq:x9wave}
\end{equation}
with the branch taken by continuity from $\theta=0$.

\emph{Thirteenth order.} Here $r=3$, the lowest degree is four by Lemma~\ref{lem:equal}, and the
admissible patterns of magnitude sum four are $(1,-1,2)$ and $(2,-1,1)$, since \eqref{eq:h2bound}
excludes $(1,-2,1)$. Nonemptiness is certified as follows. Fix $\lam_1=1/2$ and set
$\Phi(\lam_2,\lam_3)=(h_2+h_1,h_3-2h_1)$, with the $h_k$ from \eqref{eq:arctan}. On the box
\begin{equation}
    -0.127\le\lam_2\le-0.126,\qquad -0.184\le\lam_3\le-0.183 ,
\label{eq:x13box}
\end{equation}
take the center of \eqref{eq:x13box} as the reference point and the inverse of the Jacobian of $\Phi$
there as the preconditioner. Directed interval arithmetic encloses the Krawczyk operator of $\Phi$ in
\begin{equation}
    [-0.12652,-0.12614]\times[-0.18398,-0.18302] ,
\label{eq:x13image}
\end{equation}
strictly inside \eqref{eq:x13box}, so the
Krawczyk test gives a unique zero there, at $\lam_2\simeq-0.126330$ and
$\lam_3\simeq-0.183503$. At that zero
$h_2=-h_1$ and $h_3=2h_1$, so the primitive pattern is exactly $(1,-1,2)$.
The sum of the parameters is below one half, so this example also has positive linear stiffness. For this pattern
the product \eqref{eq:XY} is $(1+j\kap_1\eta)(1-j\kap_2\eta)(1+j\kap_3\eta)^2$, whose real
and imaginary parts are $X=1-c_2\eta^2+c_4\eta^4$ and $Y=c_1\eta-c_3\eta^3$, where $c_i$ is the $i$th
elementary symmetric function of the signed multiset $(\kap_1,-\kap_2,\kap_3,\kap_3)$. The inversion is
the quartic
\begin{equation}
    c_4\sigma\,\eta^4+c_3\,\eta^3-c_2\sigma\,\eta^2-c_1\,\eta+\sigma=0 ,
    \label{eq:quartic13}
\end{equation}
solved in radicals by Ferrari's formula. Writing \eqref{eq:quartic13} as
$g_4\eta^4+g_3\eta^3+g_2\eta^2+g_1\eta+g_0=0$, the depressed quantities are
\begin{equation}
    P_0=\frac{8g_4g_2-3g_3^2}{8g_4^2},\qquad Q_0=\frac{g_3^3-4g_4g_3g_2+8g_4^2g_1}{8g_4^3} ,
    \label{eq:pq13}
\end{equation}
the resolvent invariants are
\begin{equation}
\begin{split}
    \Delta_0&=g_2^2-3g_3g_1+12g_4g_0,\\
    \Delta_1&=2g_2^3-9g_3g_2g_1+27g_3^2g_0+27g_4g_1^2-72g_4g_2g_0 ,
\end{split}
    \label{eq:delta13}
\end{equation}
and, for generic $\sigma$ with $g_4R_0S_0\ne0$, with
\begin{equation}
\begin{split}
    R_0&=\Bigl(\tfrac12\bigl(\Delta_1+\sqrt{\Delta_1^2-4\Delta_0^3}\bigr)\Bigr)^{1/3},\\
    S_0&=\tfrac12\sqrt{-\tfrac23P_0+\tfrac{1}{3g_4}\bigl(R_0+\Delta_0/R_0\bigr)}
\end{split}
    \label{eq:QS13}
\end{equation}
the four roots are
\begin{equation}
    \eta=-\frac{g_3}{4g_4}\mp S_0\pm\tfrac12\sqrt{-4S_0^2-2P_0\pm\frac{Q_0}{S_0}} ,
    \label{eq:ferrari13}
\end{equation}
with the inner sign tied to the outer sign of $S$. If the displayed choice gives $Q=0$, the other
choice of sign in \eqref{eq:QS13} is used; if $S=0$ or $g_4=0$, the lower-degree limiting equation is
used instead. These are removable failures of the displayed Ferrari parametrization, not failures of
radical solvability. The physical root starts at the origin and increases according to the
continuation rule above. At $\psi=\pi/2$, division of \eqref{eq:quartic13} by $\sigma$ followed by the
limit $|\sigma|\to\infty$ gives $c_4\eta^4-c_2\eta^2+1=0$.

\emph{Seventeenth order.} Here $r=4$ and only reciprocal members are solvable. For the pattern
$(2,-1,-1,2)$, partial fractions in \eqref{eq:duxi} give a two-term law whose coefficients stand in the
ratio $2:-1$ when
\begin{equation}
\begin{split}
    \rho_2\bigl(\rho_1^2-4\bigr)&=2\rho_1\bigl(\rho_2^2-4\bigr),\\
    \rho_1&=\frac{\rho_2^2-4+\sqrt{\rho_2^4-4\rho_2^2+16}}{\rho_2} ,
\end{split}
    \label{eq:family17}
\end{equation}
a one-parameter family defined for every $\rho_2>2$, with $\rho_2<\rho_1<2\rho_2$. The merged phase is $\psi=\chi+\pi/2$ with
$\chi=2\tan^{-1}(\xi/\rho_1)-\tan^{-1}(\xi/\rho_2)$, so $\sigma=-1/\tan\chi$, and the product
\eqref{eq:XY} in the reduced variable gives
\begin{equation}
    \sigma\xi^3+(2\rho_1-\rho_2)\xi^2+\sigma\bigl(2\rho_1\rho_2-\rho_1^2\bigr)\xi+\rho_1^2\rho_2=0 .
    \label{eq:cubic17}
\end{equation}
For example, $\rho_2=2.03$ and \eqref{eq:family17} give $\rho_1\simeq2.06044$ and
\begin{equation}
(\lam_1,\lam_2,\lam_3,\lam_4)\simeq
(0.63309,0.41337,-0.29247,-0.38766),
\label{eq:x17example}
\end{equation}
whose weights are proportional to $(2,-1,-1,2)$ and for which
$2\sum_k\lam_k\simeq0.73265<1$. Thus the family is nonempty and includes members with positive linear
stiffness. Both $\rho_1<2\rho_2$ and $\rho_2<2\rho_1$ hold on the family, so
\begin{equation}
    \frac{d\chi}{d\xi}=\frac{\rho_1\rho_2\bigl(2\rho_2-\rho_1\bigr)+\bigl(2\rho_1-\rho_2\bigr)\xi^2}
    {\bigl(\rho_1^2+\xi^2\bigr)\bigl(\rho_2^2+\xi^2\bigr)}>0 ,
    \label{eq:chiprime}
\end{equation}
and $\chi$ increases from $-\pi/2$ to $\pi/2$ as $\xi$ runs along the real line. Therefore
$\sigma=-1/\tan\chi$ takes each finite nonzero real value exactly once, and for such $\sigma$ the
cubic \eqref{eq:cubic17} has exactly one real root. With
$B_2=(2\rho_1-\rho_2)/\sigma$, $B_1=2\rho_1\rho_2-\rho_1^2$, $B_0=\rho_1^2\rho_2/\sigma$
and
\begin{equation}
    C_1=B_1-\tfrac13B_2^2,\qquad
    C_0=\tfrac{2}{27}B_2^3-\tfrac13B_2B_1+B_0 ,
    \label{eq:dep17}
\end{equation}
the root is
\begin{equation}
\begin{split}
    \xi&=\sqrt[3]{-\tfrac{C_0}{2}+\sqrt{\Omega}}+\sqrt[3]{-\tfrac{C_0}{2}-\sqrt{\Omega}}
        -\frac{B_2}{3},\\
    \Omega&=\frac{C_0^2}{4}+\frac{C_1^3}{27} ,
\end{split}
    \label{eq:cardano17}
\end{equation}
where the real cube roots are taken, and $\eta=\tfrac12(\xi+\sqrt{\xi^2+4})$. At $\sigma=0$ on the
turning-point lift the real root is at infinity and \eqref{eq:cardano17} is interpreted by its limit.
For two softening factor parameters the shallower one always carries
the smaller coefficient, by the monotonicity of $|\lam|^{\,r-1}/\kap$ in
Section~\ref{sec:reciprocal}. The present construction with two pairs is sufficient for the
existence assertion in Corollary~\ref{cor:orders}; three pairs are excluded by the inequalities of
Section~\ref{sec:reciprocal}, and four or more pairs give reduced degree at least five and are
excluded by Theorem~\ref{thm:main}.

Figure~\ref{fig:waves} shows the two highest closed forms against a numerical integration of the
angle flow.

\begin{figure*}[t]
\centering
\begin{subfigure}{0.32\textwidth}\centering\includegraphics[width=\linewidth]{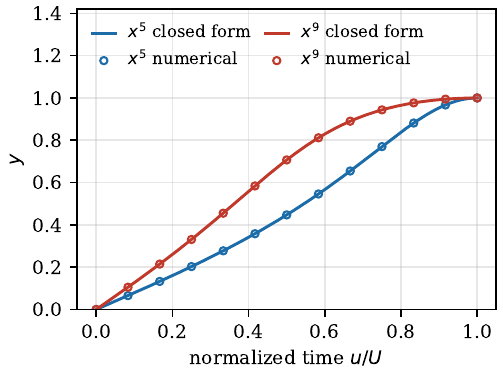}
\caption{}\end{subfigure}\hfill
\begin{subfigure}{0.32\textwidth}\centering\includegraphics[width=\linewidth]{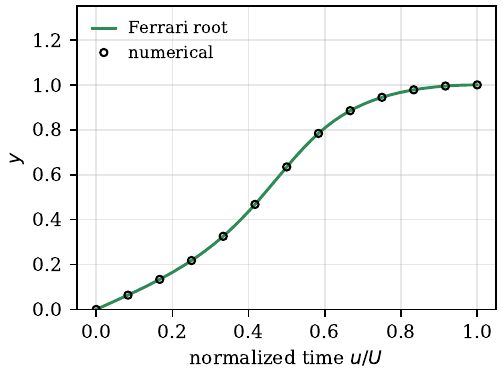}
\caption{}\end{subfigure}\hfill
\begin{subfigure}{0.32\textwidth}\centering\includegraphics[width=\linewidth]{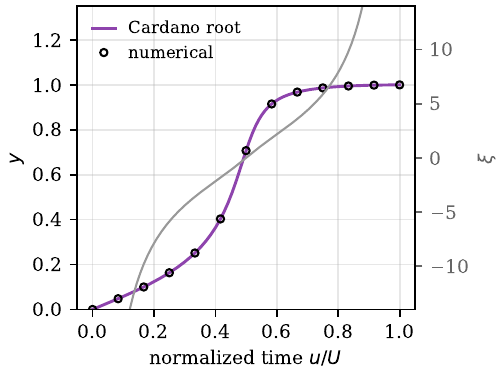}
\caption{}\end{subfigure}
\caption{Closed forms over a quarter period, with markers from a numerical integration of the angle
flow \eqref{eq:quad}. (a) The fifth and ninth orders, from \eqref{eq:x5wave} and \eqref{eq:x9wave}.
(b) A thirteenth-order member, from Ferrari's formula \eqref{eq:ferrari13}. (c) A
seventeenth-order member, from Cardano's formula \eqref{eq:cardano17}; the gray curve on the right
axis is the reduced variable $\xi$. The plots are illustrative and are not used in the proof.}
\label{fig:waves}
\end{figure*}

\section{The last configuration}
\label{sec:general}

Lemma~\ref{lem:primitive} leaves one configuration to rule out before Theorem~\ref{thm:main} holds for
every number of factor parameters, namely an indecomposable member whose inversion degree is a power of two,
$D=2^{\ell}$ with $r=D-2$, whose weights are two entries of magnitude two and $r-2$ entries of
magnitude one, and whose monodromy group is a solvable primitive subgroup of the affine group of
degree $D$ containing the $r$-cycle above $F(j)$, a single cycle of length $D-2$ with exactly two
fixed points. Lemma~\ref{lem:exceptional} shows that no such member exists, at any degree, by an
argument that needs only the affine form of the group and the cycle structure of an affine map.

The smallest members the configuration would have reached are the twenty-fifth and fifty-seventh
orders, and both are excluded along with every larger one.

\section{Consequences}
\label{sec:remarks}

\emph{Substituting a power of the displacement.} Passing to the variable $y=x^k$ cannot help. In the
equation of motion the substitution produces a term proportional to the kinetic energy and fractional
powers of $y$, so the result is not a polynomial spring. In the quadrature a change of variable does
not change the value of an integral, and elementarity is a property of the differential, not of its
name. What could change the problem is a factorization of the inversion map through a map of degree at
least two, and Lemma~\ref{lem:decomp} determines every such factorization: the only one available is
the reciprocal map of degree two. A factorization through $\eta\to\eta^k$ would need $du$ to be
invariant under $\eta\to\epsilon\eta$ with $\epsilon^k=1$, while the poles of $du$ lie on the imaginary
axis and a rotation by $\epsilon\ne\pm1$ moves them off it. In the algebra, finally, if $y=x^k$ were
expressible in radicals then so would be $x=y^{1/k}$, so a member whose group is not solvable stays
unsolvable under every relabeling of this kind.

\emph{Springs written down at will.} Any product of distinct nonzero factors
$1+\lam_kw$ with $\lam_k>-1$, squared, defines a perfect-square spring whose quadrature is elementary.
The merge criterion \eqref{eq:merge} imposes $r-1$ equations on $r$ factor parameters, so its nonempty
regular solution sets are locally one-dimensional; nonemptiness must be checked for each integer
pattern. A spring written down at random is elementary but implicit: its angle is defined by a
transcendental equation.
Among the commensurate members, Theorem~\ref{thm:main} decides which have waveforms in radicals, and
raising the degree of the polynomial never brings a formula back.

\emph{Repeated factor parameters.} A repeated factor parameter keeps the perfect square but adds a
rational function of $\eta$
to the inverse tangents in \eqref{eq:arctan}, so the inversion becomes a transcendental equation of
Kepler type rather than a polynomial one, and the question treated here does not arise.

\section{Elementary quadratures beyond the hypothesis}
\label{sec:beyond}

The example \eqref{eq:mixed} and the family \eqref{eq:mixedfam} show that the multiplicity
hypothesis of the perfect-square criterion is not a technical convenience. The mechanism behind
them is the pole structure of the quadrature differential, which also decides the smallest order
at which such members exist. Write the energy polynomial as
\begin{equation}
    P(w)=S(w)^{2}M(w) ,
    \label{eq:split}
\end{equation}
where $M$ is squarefree, a decomposition that is unique once a constant is absorbed; a zero shared
by $S$ and $M$ is a zero of $P$ of odd order at least three, and the arguments below assume that
$S$ and $M$ have no common zero. The perfect-square case is $M$ constant, and a polynomial with
zeros of mixed multiplicity is one in which $S$ and $M$ are both nonconstant. Substituting
\eqref{eq:split} into \eqref{eq:clock} and setting $B(w)=w(1-w)M(w)$ gives
\begin{equation}
    u(w)=\frac{1}{2}\int\frac{dw}{S(w)\sqrt{B(w)}} ,
    \label{eq:clockB}
\end{equation}
where $B$ is squarefree of degree $\deg M+2$ because neither $w$ nor $1-w$ divides $M$. Let
$\mathcal C$ be the smooth projective model of the curve $z^{2}=B(w)$, of genus equal to the
integer part of $(\deg B-1)/2$, let $\iota$ be the involution $(w,z)\mapsto(w,-z)$, and write
$\omega=dw/[2S(w)z]$ for the differential whose integral is \eqref{eq:clockB}. The involution
reverses the sign of $\omega$, and every pole of $\omega$ lies over a zero of $S$, a simple zero of
$S$ that is not a zero of $M$ producing one simple pole on each sheet. The theorem of Liouville,
in the form established by Ritt and sharpened by Rosenlicht and by
Risch~\cite{ritt48,rosenlicht,rosen68,risch69,risch70,bronstein}, states that an algebraic
differential has an elementary primitive precisely when it is the differential of an algebraic
function plus a constant combination of logarithmic derivatives. Applied to $\omega$ and combined
with its anti-invariance under $\iota$, it takes the following form.

\begin{theorem}
\label{thm:lr}
The integral of $\omega$ is elementary over the function field of $\mathcal C$ if and only if
there are a rational function $\varphi(w)$, nonzero functions $H_1,\dots,H_t$ on $\mathcal C$ and
constants $c_1,\dots,c_t$ with
\begin{equation}
    \omega=d\bigl[\varphi(w)\,z\bigr]+\sum_{i=1}^{t}c_i\,\frac{dH_i}{H_i},\qquad
    H_i\circ\iota=H_i^{-1} .
    \label{eq:lr}
\end{equation}
\end{theorem}

The proof is the classical one: writing $\omega$ in the Liouville form, subtracting its pullback
by $\iota$ and halving removes the invariant part, and every anti-invariant element of the
quadratic function field has the shape $\varphi(w)z$. The statement is recorded because the
remaining arguments specialize it. A nonconstant rational function on $\mathcal C$ has a pole, and
its differential has a pole of order at least two there, whereas a logarithmic derivative has only
simple poles. When $S$ has only simple zeros, none of them a zero of $B$, the differential
$\omega$ has only simple poles, so the exact term in \eqref{eq:lr} must vanish and
$\omega=\sum_{i=1}^{t}c_i\,dH_i/H_i$. The residue of a logarithmic derivative at any point is the
integer order of $H_i$ there. Choosing a basis over the rationals for the space spanned by the
residues of $\omega$ and clearing denominators in each coordinate produces integral divisors of
degree zero supported on the poles, and the quadrature is elementary exactly when some positive
integer multiple of each of them is principal, that is when each is torsion in the Jacobian, and
the holomorphic differential left after subtracting the corresponding logarithmic derivatives
vanishes. At a simple zero $a$ of $S$ with $B(a)\ne0$ the two points over $a$ carry the residues
$\pm1/[2S'(a)\sqrt{B(a)}]$. Two limiting cases place the criterion. If $M$ is constant the curve
has genus zero, nothing obstructs, and the quadrature is an inverse tangent, which is the
sufficiency of the perfect-square construction. If $S$ is constant then $\omega$ is holomorphic,
and a nonzero holomorphic differential on a compact curve of positive genus is neither exact nor
logarithmic, which is the necessity half of the criterion under the multiplicity hypothesis. Mixed
multiplicities are exactly the case in which both $S$ and $M$ are nonconstant, so that poles are
present and the torsion condition is the operative one. For the family \eqref{eq:mixedfam} the
residues at the two poles have ratio minus two and the weighted divisor of poles is principal,
which is what the pullback \eqref{eq:mixedpull} to a conic expresses. Torsion conditions of a
related kind govern the pseudo-elliptic integrals of Abel and Chebyshev~\cite{chebyshev,vdp},
where the integrand has no finite pole and the divisor in question is supported at
infinity.

\begin{proposition}
\label{prop:min}
Within the sequence in which the energy polynomial has degree $2r$ and the restoring force has
degree $4r+1$, the ninth force order is the lowest at which a nonsquare elementary member occurs,
and it is attained.
\end{proposition}

\begin{proof}
At the fifth force order the energy polynomial has degree two. If its two zeros coincide it is a
square after normalization. If they are distinct then $S$ is constant, $M$ is $P$ up to a
constant, and $B$ is a squarefree quartic whose zeros are the two zeros of $P$ together with $0$
and $1$, distinct from them because $P$ is positive on the closed swing; the curve has genus one
and $\omega$ is a nonzero holomorphic differential on it. At the ninth order take $S(w)=w+p$ with
$p>0$ and $T(w)=(w+p)^{2}$, so that $T'/T=2/S$ and
\begin{equation}
    \frac{dw}{S(w)\sqrt{B(w)}}=\frac12\,\frac{dT}{T\sqrt{B}} ,
    \label{eq:pullpower}
\end{equation}
which is a differential on a conic whenever $B$ is a quadratic polynomial in $T$. Since $B$
vanishes at $w=0$ and $w=1$ that quadratic is $-(T-p^{2})(T-(1+p)^{2})$, which gives
$M(w)=(w+2p)(w+2p+1)$ and the energy polynomial
\begin{equation}
    P(w)=\frac{(w+p)^{2}(w+2p)(w+2p+1)}{2p^{3}(2p+1)} ,
    \label{eq:ninth}
\end{equation}
with zero multiplicities two, one and one, positive on the closed swing and not a square. Its
quadrature is the single inverse tangent
\begin{equation}
    \int\frac{dT}{T\sqrt{(T-p^{2})\bigl((1+p)^{2}-T\bigr)}}
    =\frac{2}{p(1+p)}\,\tan^{-1}\!\left[\frac{1+p}{p}\sqrt{\frac{T-p^{2}}{(1+p)^{2}-T}}\,\right] ,
    \label{eq:arctan9}
\end{equation}
and the restoring force has degree nine with every coefficient positive exactly when
$p>(5+3\sqrt5)/4$, the linear coefficient being the one that binds.
\end{proof}

The construction in the proof runs unchanged with $T=(w+p)^{m}$ at every $m\ge2$ and gives a
nonsquare elementary member at every force order $4r+1$ with $r\ge2$, while the family
\eqref{eq:mixedfam} is one of the two branches of the cubic pullback $T=(w+p)^2(w+q)$, the other
being $q=4p$. The classification of the mixed-multiplicity members that this criterion admits,
through a finite critical-value condition on the pullback, is carried out in~\cite{p84}.

\section{Explicit radical orbits of the power pullback}
\label{sec:power}

The pullback $T=(w+p)^{m}$ of Section~\ref{sec:beyond} inverts in closed form at every $m$, which
turns the existence statement recorded there into an explicit waveform and settles the sign of
every force coefficient. This section also separates three questions that a single word can
conflate, namely whether the clock is algebraic at all, whether an algebraic clock inverts in
radicals, and how common the parameters that make the clock algebraic are. Time and displacement
are normalized as in \eqref{eq:energyG} throughout.

\emph{What algebraicity means here:} the first question concerns the clock alone. Fix a nonzero
real constant $h_*$ and take the branch through $u(0)=0$, so that
\begin{equation}
    Z(\eta)=\prod_{k=1}^{r}\left(\frac{1+j\kap_k\eta}{1-j\kap_k\eta}\right)^{h_k/h_*},
    \qquad \frac{Z-1}{j(Z+1)}=\tan\frac{u(\eta)}{h_*} .
    \label{eq:clockprod}
\end{equation}

\begin{proposition}
\label{prop:clock}
For the distinct-real-factor square family \eqref{eq:perfectsquare}, the germ of
$\tan[u(\eta)/h_*]$ is algebraic over $\mathbb C(\eta)$ if and only if every ratio $h_k/h_*$ is
rational, and $Z$ is a rational function of $\eta$ if and only if every such ratio is an integer.
A constant tangent clock over which the inverse displacement is algebraic therefore exists exactly
when the weights are commensurate.
\end{proposition}

\begin{proof}
At $\eta=j/\kap_k$ the local exponent of $Z$ is $h_k/h_*$, and distinctness of the amplitudes
keeps every other factor finite and nonzero there. An algebraic function has finite local
monodromy, so $\exp(2\pi jh_k/h_*)$ has finite order, which for a real exponent is rationality.
Rational exponents in \eqref{eq:clockprod} return an algebraic function and integer exponents a
rational one, while a rational function has integral zero and pole orders, which gives the second
equivalence. The relation between $Z$ and the tangent is rational in each direction, a nonconstant
algebraic relation between $\eta$ and that tangent gives algebraicity either way, and the
quadratic substitutions between $\eta$ and $y$ preserve it.
\end{proof}

\begin{proposition}
\label{prop:branch}
In the commensurate square family, suppose the monodromy group of $F=Y/X$ is nonsolvable. Then no
germ of a generic physical displacement $y(u)$ is expressible in radicals over
$\mathbb C(\sigma)$.
\end{proposition}

\begin{proof}
Coprimality of $X$ and $Y$ makes $Y(\eta)-\sigma X(\eta)$ irreducible over $\mathbb C(\sigma)$,
since a factorization in $\mathbb C[\eta,\sigma]$ of degree one in $\sigma$ would carry a factor
free of $\sigma$ dividing both $X$ and $Y$, and Gauss's lemma transfers this to the quotient
field. Its degree in $\eta$ is $D$. Over a field of characteristic zero containing the roots of
unity, the normal closure of an extension generated by a radical expression has solvable Galois
group, so a radical root of this irreducible polynomial would force the splitting field of the
whole polynomial to have solvable group, against the hypothesis. The substitutions
$y=\eta(1+\eta^{2})^{-1/2}$ and $\eta=y(1-y^{2})^{-1/2}$ are quadratic, so a radical expression
for a physical germ of $y$ away from the turning points would give one for $\eta$.
\end{proof}

The three alternatives are therefore distinct. An incommensurate clock is algebraic over no
constant tangent-clock field, a commensurate clock gives an algebraic inversion that may still
fail to be solvable by radicals, and a solvable monodromy group gives an inversion in radicals.
These are statements about function germs over the stated fields, and they neither exclude
expressions built from other special functions nor exclude simple values at isolated times. Every
regular real orbit under discussion stays periodic in all three cases.

\begin{proposition}
\label{prop:measure}
For $r\ge2$ the commensurate parameters form a set of Lebesgue measure zero and of first category
in each open component of the ordered parameter domain of \eqref{eq:perfectsquare}.
\end{proposition}

\begin{proof}
The weights $h_k$ are real analytic and nonzero on such a component, and the ratio $h_2/h_1$ is
not constant on it. If $\lam_1>0$, fix the remaining parameters and let $\lam_1\to+\infty$; then
$h_1$ decays like $\lam_1^{-1/2}$ while $h_2$ decays at least as fast as $\lam_1^{-1}$, so the
ratio tends to zero without ever vanishing. If every parameter is negative, let $\lam_1\to0^-$
instead, so that $h_1$ tends to zero to order $r-1$ while $h_2$ has a nonzero finite limit. For
each rational number the analytic set on which $h_2/h_1$ takes that value has measure zero and
empty interior, and commensurability lies in the countable union of those sets.
\end{proof}

Commensurability is in this sense arithmetically exceptional, which is a statement about exact
inversion rather than about the orbits themselves. The proposition does not assert that every
integer weight pattern is realizable, nor that the commensurate parameters are dense.

\emph{The power pullback inverted:} the construction of Section~\ref{sec:beyond} now runs
explicitly at every order.

\begin{theorem}
\label{thm:power}
Let $m\ge2$ be an integer, let $p>0$, and set
\begin{gather}
    a=p^{m},\quad b=(p+1)^{m},\quad \varrho=\frac ab,\quad K=mp^{m+1}(b-a),\quad T(w)=(w+p)^{m},
    \nonumber\\
    M(w)=\frac{[T(w)-a][b-T(w)]}{w(1-w)},\qquad
    P(w)=\frac{(w+p)^{2}M(w)}{K} .
    \label{eq:powerP}
\end{gather}
Then $P$ is a nonsquare polynomial of degree $2m$ with $P(0)=1$, positive on $[0,1]$, and the
force belonging to $G(w)=(1-w)P(w)$ has degree $4m+1$. Writing
\begin{equation}
    \psi=u\sqrt{\frac{m}{p(1-\varrho)}},\qquad
    \mathcal R=\bigl(\cos^{2}\psi+\varrho\sin^{2}\psi\bigr)^{1/m} ,
    \label{eq:powerangle}
\end{equation}
the orbit with $y(0)=0$ and $y'(0)=1$ is
\begin{equation}
    y=\sqrt{p(1-\varrho)}\,\sin\psi
    \Bigl[\,\mathcal R\sum_{k=0}^{m-1}\mathcal R^{\,k}\Bigr]^{-1/2} ,
    \label{eq:powerwave}
\end{equation}
with positive roots throughout. It is real analytic for every real $u$, has amplitude one and
minimal period $2\pi\sqrt{p(1-\varrho)/m}$, and is radical over $\mathbb C(\tan\psi)$.
\end{theorem}

\begin{proof}
The two numerator factors of $M$ vanish simply at $w=0$ and at $w=1$, so the quotient is a
polynomial, and $T'=m(w+p)^{m-1}$ vanishes only at $w=-p$, where $T=0$ is neither $a$ nor $b$.
Hence $M$ is squarefree of degree $2m-2$ and nonzero at $-p$, so $P$ carries one double zero and
$2m-2$ simple ones. The unnormalized numerator takes the value $K$ at $w=0$, and monotonicity of
$T$ on $[0,1]$ gives positivity there, endpoints included. On the first quarter orbit the
normalized clock is
\begin{equation}
    du=\frac{\sqrt K}{2m}\,\frac{dT}{T\sqrt{(T-a)(b-T)}} ,
    \label{eq:powerclock}
\end{equation}
whose inversion on the conic is $T=ab[b\cos^{2}\psi+a\sin^{2}\psi]^{-1}$ with
$\psi=m\sqrt{ab}\,u/\sqrt K$, which is \eqref{eq:powerangle}. This gives
$y^{2}=p(\mathcal R^{-1}-1)$, and rationalizing $1-\mathcal R$ through
$1-\mathcal R^{m}=(1-\varrho)\sin^{2}\psi$ produces \eqref{eq:powerwave}. That expression has no
vanishing denominator, carries the correct sign across successive quarter periods, and has
derivative one at the origin. The energy identity of the first quarter extends analytically to all
real times, which gives the equation of motion. The positive excursion occupies $0<\psi<\pi$ with
its single maximum at $\psi=\pi/2$ and the negative excursion the next half cycle, which fixes the
minimal period. The inversion equation is
\begin{equation}
    (1+\varrho\tan^{2}\psi)(p+y^{2})^{m}-p^{m}(1+\tan^{2}\psi)=0 ,
    \label{eq:powerinverse}
\end{equation}
and one $m$th root followed by square roots expresses each of its roots, the displayed formula
being its continuous physical branch.
\end{proof}

At $m=2$ the energy polynomial \eqref{eq:powerP} is \eqref{eq:ninth}, and at $m=5$ it already
supplies a radical orbit for a force of degree twenty-one. The statement concerns the design orbit
and does not say that every orbit of the same spring is elementary. It stands beside
Corollary~\ref{cor:orders} rather than against it, since that corollary governs the
distinct-real-factor square family while \eqref{eq:powerP} is nonsquare.

\begin{corollary}
\label{cor:power}
If $p\ge2m$, every coefficient of the odd polynomial force belonging to \eqref{eq:powerP} is
strictly positive.
\end{corollary}

\begin{proof}
Factor $M=A_1A_2$ with
\begin{equation}
\begin{split}
    A_1(w)&=\frac{(w+p)^{m}-p^{m}}{w}=\sum_{k=0}^{m-1}p^{k}(w+p)^{m-1-k},\\
    A_2(w)&=\frac{(p+1)^{m}-(w+p)^{m}}{1-w}=\sum_{k=0}^{m-1}(p+1)^{m-1-k}(w+p)^{k} .
\end{split}
    \label{eq:powerfactor}
\end{equation}
The numerator of $P$ is then a polynomial in $w+p$ with nonnegative coefficients, of degree $2m$,
carrying at least two nonzero coefficients of different degrees. Writing
$P(w)=\sum_{k=0}^{2m}d_kw^{k}$ and comparing consecutive coefficients inside each shifted monomial
gives $0<d_{k+1}/d_k\le(2m-k)/[(k+1)p]$ for $0\le k<2m$. For $p\ge2m$ that ratio is below one when
$k\ge1$, and also when $k=0$, including at $p=2m$, because some nonzero coefficient of the shifted
numerator has degree below $2m$. With $d_0=1$ and $d_{2m+1}=0$ the force is
$\sum_{k=0}^{2m}(k+1)(d_k-d_{k+1})y^{2k+1}$ with every difference positive.
\end{proof}

The bound is sufficient and not necessary, and it does not replace the sharper ninth-order bound of
Proposition~\ref{prop:min}. What it gives is an explicit infinite sequence of confining, strictly
hardening springs whose design orbit is radical.

One quantity attached to the design orbit is rational whenever the factor parameters are, whether
or not the clock is commensurate and whether or not its inversion is radical. On the upper half of
the design orbit the velocity is $\sqrt{1-y^{2}}\prod_{k=1}^{r}(1+\lam_ky^{2})$, with every factor
positive, so writing $\prod_{k=1}^{r}(1+\lam_kw)=\sum_{k=0}^{r}d_kw^{k}$ and using
$\int_{0}^{1}y^{2k}\sqrt{1-y^{2}}\,dy=\pi C_k/4^{k+1}$ with $C_k=\binom{2k}{k}/(k+1)$ gives the
action of the design orbit as
\begin{equation}
    \frac{1}{2\pi}\oint y'\,dy=\frac12\sum_{k=0}^{r}\frac{C_kd_k}{4^{k}} .
    \label{eq:designaction}
\end{equation}

\section{An elementary route through the group-theoretic step}
\label{sec:verify}

The exclusion of a solvable monodromy group in Section~\ref{sec:groups} rests on a bound for the
order of an element of a general linear group. That bound can be avoided. The argument below
reaches the same conclusion from a fixed-space obstruction whose proof is self-contained, which
makes the classification independent of any element-order bound. The affine restriction used in it
is classical and is also a consequence of the cycle classification of~\cite{jones}; no
classification of finite groups is claimed here.

\begin{lemma}
\label{lem:affine}
Let $q$ be a prime and let an affine permutation of a finite vector space $V$ over $\mathbb F_q$
have exactly one nontrivial cycle, of length $s>1$, together with $f\ge2$ fixed points. Then
$q=2$, the space is two dimensional, and $f=s=2$.
\end{lemma}

\begin{proof}
Translating a fixed point to the origin makes the permutation a linear map $A$. Put
$U=\ker(A-I)$, so that $U$ has $f\ge2$ elements, and let $V_1$ be its complement in $V$, on which
$A$ restricts to a single $s$-cycle. Translation by an element of $U$ preserves $V_1$ and commutes
with $A$, and it acts faithfully on $V_1$ because a nonzero translation has no fixed point. The
centralizer of an $s$-cycle in the symmetric group on its support is cyclic of order $s$, since a
commuting permutation is determined by the image of one point and is therefore a power of the
cycle. The elementary abelian translation group of $U$ thus embeds in a cyclic group, which forces
$U$ to be one dimensional, $f=q$, and $q$ to divide $s$.

The subgroup of order $q$ inside that cyclic centralizer is unique, so the restriction of
$A^{s/q}$ to $V_1$ is translation by some nonzero element of $U$, while on $U$ that power is the
identity. The linear map $L=A^{s/q}-I$ therefore vanishes on $U$ and is constant and nonzero on
$V_1$, so its image has exactly two elements. An image is a subspace over $\mathbb F_q$, which
leaves only $q=2$. The rank of $L$ is then one and its kernel is exactly $U$, of dimension one, so
rank and nullity give dimension two. The values of $f$ and $s$ follow.
\end{proof}

\begin{proposition}
\label{prop:primitive}
An indecomposable inversion map of the distinct-factor square family cannot have solvable
monodromy when $D\ge5$.
\end{proposition}

\begin{proof}
Write $\Gamma$ for the monodromy group. A rational map is indecomposable exactly when $\Gamma$ is
primitive, since a nontrivial block system gives an intermediate function field, rational by
L\"uroth's theorem and hence a rational decomposition, while the fibers of an inner map give the
converse~\cite{ritt22,volklein}. A solvable faithful primitive group is affine. Indeed a minimal
nontrivial normal subgroup $\mathcal N$ is elementary abelian, because its derived subgroup, its
primary components and its subgroup of elements killed by a prime are all characteristic; the
orbits of $\mathcal N$ are blocks, so $\mathcal N$ is transitive; a point stabilizer in an abelian
transitive subgroup fixes every point and is trivial by faithfulness. Thus $\mathcal N$ is
regular, identifies the underlying set with a vector space over a prime field, and its normalizer
acts by affine maps~\cite{dixon}.

The branch data supply an $r$-cycle with $f=D-r$ fixed points, and $r\ge2$ because $r=1$ has
primitive degree one. The case $f=0$ is excluded for $D\ge5$ by Lemma~\ref{lem:equal}. If
$f\ge2$, Lemma~\ref{lem:affine} forces $D=4$. If $f=1$, the permutation over the marked value $j$
has $D-1$ cycles on $D$ points, so its cycle-length defect is one and it is a transposition. A
primitive permutation group containing a transposition is the full symmetric group, since the
graph whose edges are the conjugates of that transposition has components forming blocks,
primitivity makes the graph connected, and the edge transpositions of a connected graph generate
the symmetric group. For $D\ge5$ that group is not solvable, and every case is excluded.
\end{proof}

The classification follows from this without an element-order bound. If the monodromy is solvable
and $D\ge5$, Proposition~\ref{prop:primitive} forces a decomposition, which
Lemma~\ref{lem:decomp} makes reciprocal with inner degree two, and
Section~\ref{sec:reciprocal} identifies the outer map, up to a target M\"obius transformation,
with the reduced member of degree $D/2$. The monodromy of that reduced member is a quotient of
$\Gamma$ and is solvable, and the reduced member carries softening amplitudes only, so it is not
reciprocal and the same decomposition argument makes it indecomposable;
Proposition~\ref{prop:primitive} then bounds its degree by four. Conversely a reciprocal map whose
reduced degree is at most four is solvable, because the block action of its monodromy is the
reduced monodromy group and the kernel of that action embeds in a direct product of groups of
order two acting inside the two-point blocks, so the extension is solvable. Equivalently, solve the
reduced equation in radicals and then solve $\eta^{2}-\xi\eta-1=0$. This gives both directions of
Theorem~\ref{thm:main} and an independent exclusion of the affine configuration of
Lemma~\ref{lem:exceptional}.

\section{Conclusion}

A commensurate distinct-factor member of the perfect-square family has an inversion in radicals over
$\mathbb C(\sigma)$ exactly when its primitive inversion degree is at most four, or when its amplitudes
form reciprocal pairs and the reduced member has primitive degree at most four. The reduction applies
once and no more, since a reduced member carries softening amplitudes only and can never itself be
reciprocal. The criterion therefore leaves radical orbits at the fifth, ninth and thirteenth force
orders, and at the seventeenth only for reciprocal members. The mixed-multiplicity example
\eqref{eq:mixed} and the one-parameter family it belongs to show why this is not an exhaustive
classification of all elementary polynomial springs, and Section~\ref{sec:beyond} places those
members: their quadrature is a differential of the third kind whose elementarity is a torsion
condition on the divisor of its poles together with a vanishing holomorphic remainder, and within
the orders congruent to one modulo four they begin at the ninth. Theorem~\ref{thm:power} inverts
the power pullback explicitly at every one of those orders and Corollary~\ref{cor:power} fixes
the sign of every force coefficient, so radical design orbits occur for confining springs of
arbitrarily high degree outside the square family. Within that family the clock is algebraic only
for commensurate weights, which are exceptional in the parameter domain, and an algebraic clock
still need not invert in radicals.

Explicit periodic orbits are rare among nonlinear oscillators, and whether one exists is usually
settled case by case. For this family the question has an algebraic answer, and the answer marks
where the radical inversions stop inside the stated class. Past that boundary the selected libration
stays periodic and the quadrature stays elementary. Only the step from time back to displacement
ceases to be solvable by radicals over the specified base field.

\begin{acknowledgments}
AI tools were used mainly for editorial assistance including proofreading, literature search and
language refinement. All theoretical developments, derivations and conclusions are the author's work.
\end{acknowledgments}

\section*{Data availability}
No data were created or analyzed in this study.

\end{document}